\documentclass[12pt,reqno]{amsart}
\usepackage{amssymb}
\usepackage{longtable}
\theoremstyle{plain}
\newtheorem{theorem}{Theorem}
\newtheorem{lemma}{Lemma}

\newtheorem{proposition}{Proposition}[section]
\theoremstyle{proof}
\theoremstyle{definition}

\theoremstyle{remark}

\theoremstyle{lamma}

\numberwithin{equation}{section}
\numberwithin{lemma}{section}
\numberwithin{theorem}{section}

\usepackage{amssymb, amsmath, amsthm}
\usepackage[breaklinks]{hyperref}
\theoremstyle{thmrm}

\usepackage{graphicx}
\begin{document}
\title[]{Irreducible Modules for the Loop of Derivations of Rational Quantum Torus}
\author[]{Santanu Tantubay}
\address{Harish-Chandra Research Institute, HBNI,
Chhatnag Road, Jhunsi, Prayagraj 211 019, India.}
\email{santanutantubay@hri.res.in,tantubaysantanu@gmail.com}
 
\maketitle
\begin{abstract}
Let $\mathbb{C}_q$ be a rational quantum torus with the matrix $q$. Let $Der (\mathbb{C}_q)$ be the Lie algebra of derivations of $\mathbb{C}_q$. In this paper we consider the Lie algebra $(\mathbb{C}_q\rtimes Der(\mathbb{C}_q))\otimes B$, where $B$ is commutative associative unital algebra over $\mathbb{C}$ and classify its irreducible modules with finite dimensional weight spaces.   
\end{abstract}
\section{Introduction}
Let $q=(q_{i,j})$ be a $n\times n$ matrix such that $q_{i,i}=1$ and $q_{i,j}=q_{j,i}^{-1}$ for $1\leq i,j\leq n$. Let $\mathbb{C}_q$ be the Laurent polynomial ring in $n$ non-commutative variables $t_1, \dots, t_n$ with the conditions $t_it_j=q_{i,j}t_jt_i$ and all $q_{ij}$'s are roots of unity. It is easy to see that if $q$ is the identity matrix, then $\mathbb{C}_q$ becomes the Laurent polynomial ring $A=\mathbb{C}[t_1^{\pm 1},t_2^{\pm 1},\dots, t_n^{\pm 1}]$.  Let $Der (\mathbb{C}_q)$ be the Lie algebra of derivations of $\mathbb{C}_q$. Therefore $Der(\mathbb{C}_q)$ naturally acts on $\mathbb{C}_q$ and hence  $\mathbb{C}_q\rtimes Der(\mathbb{C}_q)$ becomes a Lie algebra. In \cite{[5]}, S. E. Rao classified irreducible modules with finite dimensional weight spaces for the Lie algebra $A\rtimes Der(A)$ with associative action of $A$. In \cite{[4]}, Priyansu Chakraborty, S. Eswara Rao consider the loop of $A\rtimes Der(A)$ and classified irreducible modules with finite dimensional weight spaces when $A\otimes B$ acts associatively. In \cite{[3]}, S. E. Rao, Punita Batra, Sachin S. Sharma consider the Lie algebra $\mathbb{C}_q\rtimes Der(\mathbb{C}_q)$ and they classified all irreducible modules with finite dimensional weight spaces. Under some conditions these modules are of the form $V\otimes \mathbb{C}_q$, where $V$ is a finite dimensional irreducible $gl_n$-module. In this paper we consider the loop of $\mathbb{C}_q\rtimes Der(\mathbb{C}_q)$ and prove that under the similar condition all irreducible modules with finite dimensional weight spaces are of the form $V\otimes \mathbb{C}_q$. For more about the representations of loop of some well known Lie algebras one can see the referrences of \cite{[4]}.
\vspace{5mm}

The paper is organised as follows. In Section \ref{Section 2} we begin with definitions and properties of the rational quantum torus $\mathbb{C}_q$. We define the Lie algebra $\tau$, the loop of $\mathbb{C}_q\rtimes Der(\mathbb{C}_q)$. Then we racall a theorem of \cite{[3]} and construct an irreducible module $F^{\alpha}(V,\phi)$ for $\tau$. In Section \ref{Section 3} we consider $V$ an irreducible $\tau$-module with finite dimensional weight spaces and assume the actions of $\mathbb{C}_q^{(1)}\otimes B$ and $\mathbb{C}_q^{(2)}\otimes B$  on $V$ are associative and anti-associative respectively. Now from the associative action of $\mathbb{C}_q^{(1)}\otimes B$ we will get an algebra homomorphism $\psi:B\rightarrow \mathbb{C}$. Then we get $ad\;t^sb=\psi(b)ad \;t^s$ and $D(u,0)b=\psi(b)D(u,0)$ on $V$. Then we prove our main result $V\cong F^{\alpha}(V,\psi)$.
\section{Notation and Preliminaries}\label{Section 2}
\begin{center}
\end{center}
In this paper all the vector spaces, algebras, tensor products are over the field of complex numbers $\mathbb{C}$. Let $\mathbb{Z},\mathbb{N}$ denote the sets of integers and natural numbers respectively. For any Lie algebra $L$, let $U(L)$ denote the universal enveloping algebra of $L$.
Let us fix a positive integer $n\geqslant2.$ Let $q=(q_{i,j})$ be a $n\times n$ matrix such that $q_{i,i}=1$ and $q_{i,j}=q_{j,i}^{-1}$ for $1\leq i,j\leq n$. Let $\mathbb{C}_q$ be the Laurent polynomial ring in $n$ non-commutative variables $t_1, \dots, t_n$ with the conditions $t_it_j=q_{i,j}t_jt_i$. Clearly we can see that $\mathbb{C}_q$ is $\mathbb{Z}^{n}$-graded with each graded component is one dimensional. For $a=(a_1,\dots, a_n) \in \mathbb{Z}^{n}$, let $t^a=t_1^{a_1}t_2^{a_2}\cdots t_n^{a_n} \in \mathbb{C}_q$.\\ Now let us define the following maps $\sigma, f :\mathbb{Z}^{n}\times \mathbb{Z}^{n} \rightarrow \mathbb{C}^*$ by
\begin{equation}
\sigma (a,b)= \prod _{1\leq i\leq j\leq n} q_{j,i}^{a_jb_i}
\end{equation}
\begin{equation}
f(a,b)=\sigma (a,b) \sigma (b,a)^{-1}
\end{equation} 
Then one has the following results for any $a,b,c \in \mathbb{Z}^{n}, k \in \mathbb{Z}$ :
\begin{enumerate}
\item $f(a,b)=f(b,a)^{-1}$,
\item $f(ka,a)=f(a,ka)=1$,
\item $f(a+b,c)=f(a,c)f(b,c)$,
\item $f(a,b+c)=f(a,b)f(a,c)$,
\item $\sigma (a,b+c)=\sigma (a,b)\sigma (a,c)$,
\item $t^at^b=\sigma(a,b)t^{a+b}$, $t^at^b=f(a,b)t^bt^a$.
\end{enumerate}
The radical of $f$ is defined by 
\begin{center}
$\text{rad} \;f=\{a \in \mathbb{Z}^{n}|f(a,b)=1, \forall b \in \mathbb{Z}^{n}\}$.
\end{center}
 It is easy to see that rad $f$ is a subgroup of $\mathbb{Z}^{n}$ and $m \in \text{rad}\; f$ iff $f(r,s)=1,\; \forall r,s \in \mathbb{Z}^{n}$ with $r+s=m$.
 \begin{proposition}
\begin{enumerate}
\item The center $Z(\mathbb{C}_q)$ of $\mathbb{C}_q$ has a basis consisting of monomials $t^a$, $a \in \text{rad}\; f $.
\item The Lie subalgebra $[\mathbb{C}_q,\mathbb{C}_q]$ of $\mathbb{C}_q$ has a basis consisting of monomial $t^a, a\in \mathbb{Z}^{n} \setminus \text{rad}\; f$.
\item $\mathbb{C}_q=[\mathbb{C}_q,\mathbb{C}_q]\oplus Z(\mathbb{C}_q)$.
\end{enumerate}
\end{proposition}
Now $\mathbb{C}_q$ is a Lie algebra with the Lie brackets $[t^a,t^b]=(\sigma (a,b)-\sigma(b,a))t^{a+b}$.
Let $Der( \mathbb{C}_q)$ be the space of all derivations of $\mathbb{C}_q$. Then we have the lemma
\begin{lemma}(See \cite{[1]}, Lemma 2.48)
\begin{enumerate}
\item Der$(\mathbb{C}_q)=\oplus_{a\in \mathbb{Z}^{n}}(\text{Der}(\mathbb{C}_q))_a$.
\item \[ \text{Der}(\mathbb{C}_q)=\begin{cases} \mathbb{C} \text{ad}\; t^a & \text{if}\;a \notin \text{rad}\;(f)\\ \bigoplus_{i=1}^{n}\mathbb{C}t^a  \partial_i & \text{if}\; a \in \text{rad}\;(f).\end{cases}\]
\end{enumerate}
\end{lemma} 
The space Der($\mathbb{C}_q$) is a Lie algebra with the following brackets:
\begin{enumerate}
\item $[\text{ad}\;t^s.\text{ad}\;t^s]=(\sigma (s,r)-\sigma (r,s))\text{ad}\; t^{s+r},\: \forall r,s \notin \text{rad} \;f$;
\item $[D(u,r),\text{ad} \;t^s]=(u,s)\sigma (r,s)\text{ad}\; t^{r+s},\: \forall r\in \text{rad}\;(f), s \notin \text{rad} (f), u \in \mathbb{C}^{n}$;
\item $[D(u,r),D(u^{\prime},r^{\prime})]=D(w, r+r^{\prime}),\; \forall r,r^{\prime}\in \text{rad}\;(f), u,u^{\prime}\in \mathbb{C}^{n}$ and where $w=\sigma(r,r^\prime)((u,r^{\prime})-(u^{\prime},r)).$
\end{enumerate}
 
Now consider the Lie algebra $\mathfrak{g}=\mathbb{C}_q\rtimes \text{Der} \;\mathbb{C}_q$ with the Lie barckets as above and
\begin{enumerate}
\item $[D(u,r),t^s]=(u,s)\sigma(r,s)t^{r+s}$, for all $r \in \text{rad}\; f, s\in \mathbb{Z}^n, u \in \mathbb{C}^n$;
\item $[ad\; t^r, t^s]=(\sigma (r,s)-\sigma (s,r))t^{r+s}, \;\forall r\notin \text{rad} (f), s\in \mathbb{Z}^n$;
\end{enumerate}
Let $\mathfrak{h}=\{D(u,0)|u \in \mathbb{C}^n\}\oplus \mathbb{C}$. Then $\mathfrak{h}$ is a maximal abelian subalgebra of $\mathfrak{g}$. Let $W$ be the Lie subalgebra of Der $\mathbb{C}_q$ generated by elements $D(u,r)$
, where  $u \in \mathbb{C}^n$, $r \in \text{rad}\; (f)$. Let $\mathbb{C}_q^{(1)}=\mathbb{C}_q$ and $\mathbb{C}_q^{(2)}=\text{span}\{ad\; t^r-t^r|r \in \mathbb{Z}^n\}$ be the subalgebras of $\mathfrak{g}$. Then as in \cite{[2]}, we can say that $\mathfrak{g}=W\ltimes (\mathbb{C}_q^{(1)}+\mathbb{C}_q^{(2)})$ and $\mathbb{C}_q^{(2)}$ is isomorphic to $\mathbb{C}_q^{(1)}$ by the Lie algebra isomorphism $ad\; t^r-t^r\rightarrow t^r$.
Let $B$ be a commutative associative unital algebra over $\mathbb{C}$. Consider $\tau=(\mathbb{C}_q\rtimes \text{Der} \;\mathbb{C}_q) \otimes B$ and define a Lie algebra structure on $\tau$ by
\begin{center}0
$[x\otimes a,y\otimes b]=[x,y]\otimes(ab),$
\end{center}
where $x,y \in \mathbb{C}_q\rtimes \text{Der} \;\mathbb{C}_q, a,b \in B$.
\begin{theorem}({\cite{[3]}}, Theorem 2.6)\label{theorem 2.1}
Let $V^{\prime}$ be an irreducible $\mathbb{Z}^n$-graded $\mathfrak{g}-module$ with finite dimensional weight spaces with respect to $\mathfrak{h}$, with associative $\mathbb{C}_q^{(1)}$ and anti-associative $\mathbb{C}_q^{(2)}$ action and $t^0=1.$ Then $V^{\prime}\cong F^{\alpha}(V)$ for some $\alpha\in \mathbb{C}^n$ and a finite dimensional irreducible $\mathfrak{gl}_n$-module $V$.  
\end{theorem}
Let $\phi: B\rightarrow \mathbb{C}$ be an algebra homomorphism such that $\phi (1)=1$. Let $V$ be an irreducible finite dimensional $\mathfrak{gl}_n(\mathbb{C})$-module. Let us define a $\tau$-module structure on $F^{\alpha}(V)=V\otimes \mathbb{C}_q$ by
\begin{enumerate}
\item $ad\; t^r \otimes b.v(s)=\phi (b)(\sigma(r,s)-\sigma(s,r))v(r+s)$, where $r\notin \text{rad}\;(f), s \in \mathbb{Z}^n$ ;
\item $D(u,r)\otimes b. v(s)=\phi(b)\sigma (r,s)\{(u,s+\alpha)v+\sum_{i,j}r_iu_jE_{ij}\}\otimes v(r+s)$, where $r \in \text{rad}\; (f), s \in \mathbb{Z}^n, u\in \mathbb{C}^n$;
\item $t^r\otimes b.v(s)=\phi (b) \sigma(r,s)v(r+s)$, where $r,s \in \mathbb{Z}^n$.
\end{enumerate}
 
Now by Proposition 2.4 of \cite{[3]}, we can see that $F^{\alpha}(V)$ is an irreducible module for $\tau$. Let us denote this $\tau$-module by $F^{\alpha}(V,\phi).$
\vspace{5mm}

In this paper we want to prove under some natural conditions every irreducible modules for $\tau$ will be of the form $F^{\alpha}(V,\phi)$.
  
\section{}\label{Section 3}

Let $V$ be an irreducible $\tau$-module with finite dimensional weight spaces. Assume that the actions of $\mathbb{C}_q^{(1)}\otimes B$ and $\mathbb{C}_q^{(2)}\otimes B$ are associative and anti-associative respectively and $t^0$ acts as identity on $V$. Let $V=\oplus_{r \in \mathbb{Z}^n}V_r$ be its weight space decomposition with $V_r=\{v\in V |D(u,0)v=(u,r+\alpha)v, \forall u \in \mathbb{C}^n\}$ for some $\alpha \in \mathbb{C}^n$. Since $t^0$ acts as identity on $V$, each weight spaces will be of same dimension. Now $1\otimes B$ is in the center of $\tau$, therefore it will act scalarly on $V$. Since the action of $\mathbb{C}_q^{(1)}\otimes B$ is associative on $V$, there exists an algebra homomorphism $\psi:B \rightarrow \mathbb{C}$ such that $1\otimes b.v=\psi (b)v,\; \forall v \in V$. Let $M$ be the kernel of this homomorphism, therefore $M$ will be a maximal ideal of $B$. We will denote an element $x\otimes b$ of $\tau$ by $xb$.
\vspace*{.2cm}

Let $U(\tau)$ denote the universal enveloping algebra of $\tau$. Let $L(\tau)$ be a two sided ideal of $U(\tau)$ generated by $\{t^rb_1t^sb_2-\sigma (r,s)t^{r+s}b_1b_2 \;\forall r,s \in \mathbb{Z}^n, (ad\;(t^r)-(t^r))b_1(ad\;(t^s)-t^s)b_2+\sigma (s,r)(ad\;t^{r+s}-t^{r+s})b_1b_2, \forall r ,s \in \mathbb{Z}^n,t^0\otimes 1-1\}$. Therefore $L(\tau)$ acts trivially on $V$ and is a $U(\tau)/L(\tau)$-module.
\begin{proposition}\label{proposition 3.1}
For $r,s \in \mathbb{Z}^n\setminus \textit{rad}\;f$ and $b_1,b_2,b_3,b_4 \in B$, we have $[t^{-s}b_1 \;ad\; t^sb_2,t^{-r}b_3\;ad\; t^rb_4]=0$ on $V$.
\end{proposition}
\begin{proof}
First using the associavity and anti-associavity of respective rational quantum tori we get
\begin{equation}
ad\; t^sb_2.ad\; t^rb_4-(t^sb_2.ad\;t^rb_4+t^rb_4.ad\;t^sb_2)+\sigma(r,s) ad\;t^{r+s}b_2b_4=0\: \textit{on}\:
V
\end{equation}

Let us consider \\
 
$[t^{-s}b_1\;ad\; t^sb_2,t^{-r}b_3\;ad\;t^rb_4]$\\
$=[t^{-s}b_1,t^{-r}b_3]\;ad\;t^s\; ad\;t^rb_4+t^{-s}b_1[ad\;t^sb_2,t^{-r}b_3]\;ad\;t^rb_4$\\ $+t^{-r}b_3\;[t^{-s}b_1,ad\;t^rb_4]\;ad\;t^sb_2+t^{-r}b_3\;t^{-s}b_1[ad\;t^sb_2,ad \;t^rb_4]$\\
$=(\sigma(s,r)-\sigma(r,s))t^{-(s+r)}b_1b_3.\{t^sb_2\;ad\;t^rb_4+t^rb_4\;ad\;t^sb_2-\sigma(r,s)\;ad\;t^{r+s}b_2b_4\}$\\$+t^{-s}b_1(\sigma(s,-r)-\sigma(-r,s))t^{s-r}b_2b_3\;ad\;t^rb_4-t^{-r}b_3(\sigma(r,-s)-\sigma(-s,r))t^{r-s}b_1b_4\;ad\;t^sb_2$\\$+\sigma(r,s)t^{-(r+s)}b_1b_3(\sigma(s,r)-\sigma(r,s))\;ad\;t^{r+s}b_2b_4$\\
$=(\sigma(s,r)-\sigma(r,s))\{\sigma(-(r+s),s)t^{-r}b_1b_2b_3\;ad\;t^rb_4+\sigma(-(r+s),r)t^{-s}b_1b_3b_4\;ad\;t^s-\sigma(r,s)t^{-(r+s)}b_1b_3\;ad\;t^{r+s}b_2b_4\}+(\sigma(s,-r)-\sigma(-r,s))\sigma(-s,s-r)t^{-r}b_1b_2b_3\;ad\;t^rb_4-(\sigma(r,-s)-\sigma(-s,r))\sigma(-r,r-s)t^{-s}b_1b_3b_4\;ad\;t^s+\sigma(r,s)(\sigma(s,r)-\sigma(r,s))t^{-(s+r)}b_1b_3\;ad\;t^{s+r}b_2b_4$\\
$=0.$
\end{proof}
\begin{lemma}\label{lemma 3.1}(\cite{[6]}) 
Let $\mathfrak{g}$ be any Lie algebra. Assume $(V^{\prime},\rho)$ be an irreducible finite dimensional module for $\mathfrak{g}$. Then $\rho (\mathfrak{g})$ is a reductive Lie algebra whose center is atmost one dimensional.
\end{lemma}

Let $U_1=U(\tau)/L(\tau)$ and let us define $T(u,r,b_1,b_2)=\sigma(r,r)t^{-r}b_1D(u,r)b_2$ as an element of $U_1$ for $r \in \textit{rad}\;f,u \in \mathbb{C}^n,b_1,b_2 \in B$. Let $T=\textit{span}\{T(u,r,b_1,b_2):r \in \textit{rad}\;f,u \in \mathbb{C}^n,b_1,b_2 \in B\}$. Let $\mathfrak{g}$ be the Lie subalgebra generated by $T(u,r,b_1,b_2)$ and $t^{-s}b_3\;ad\;t^sb_4$ for all $u\in \mathbb{C}^n; r \in \textit{rad}\;(f); s\notin \textit{rad}\;(f);b_1,b_2,b_3,b_4 \in B$. Let $I$ be the Lie subalgebra generated by the elements of the form $t^{-s}b_1\;ad\; t^sb_2$ . Then we have
\begin{lemma}\label{lemma 3.2}
$I$ be an abelian ideal of $\mathfrak{g}$.  
\end{lemma}
\begin{proof}
By Proposition \ref{proposition 3.1}, we have $I$ is an abelian subalgebra.\\

Now consider \\
$[t^{-r}b_1D(u,r)b_2,t^{-s}b_3\;ad\;t^sb_4]$\\
$=[t^{-r}b_1D(u,r)b_2,t^{-s}b_3]\;ad\;t^sb_4+t^{-s}b_3[t^{-r}b_1D(u,r)b_2,ad\;t^sb_4 ]$\\
$=t^{-r}b_1[D(u,r)b_2,t^{-s}b_3]\;ad\;t^sb_4+ t^{-s}b_3t^{-r}b_1[D(u,r)b_2,ad\;t^sb_4]\:$ (The other two terms will vanish since $r\in \textit{rad}\;f)$\\
$=-(u,s)\sigma(r,-s)t^{-r}b_1t^{r-s}b_2b_3\;ad\;t^sb_4+(u,s)\sigma(r,s)^2t^{-(r+s)}b_1b_3\;ad\;t^{r+s}b_2b_4$\\
$=(u,s)\{\sigma(r,s)^2t^{-(r+s)}b_1b_3\;ad\;t^{r+s}b_2b_4+\sigma(-r,r)t^{-s}b_1b_2b_3\;ad\;t^sb_4\}\in I$\\
 Hence we have the Lemma.
\end{proof}
\begin{proposition}\label{proposition 3.2}
\begin{enumerate}
\item $T $ is a Lie algebra with the Lie brackets\\
$[T(u,r,b_1,b_2),T(v,s,b_3,b_4)]=(v,r)T(u,r,b_1,b_2b_3b_4)-(u,s) T(v,s,b_3,b_1b_2b_4)+ T(w,r+s,b_1b_3,b_2b_4),$ where $w= ((u,s)v-(v,r)u)$.
\item $[D(v,0),T(u,r,b_1,b_2)]=0$  
\item Let $V=\oplus _{r\in \mathbb{Z}^n}V_r$ be its weight space decomposition. Then each $V_r$ is $T$ invariant.
\item Each $V_r$ is $T$-irreducible.
\end{enumerate}
\end{proposition}
\begin{proof}
\begin{enumerate}
\item $[T(u,r,b_1,b_2),T(v,s,b_3,b_4)]$\\
 $=\sigma (r,r)\sigma (s,s)[t^{-r}b_1D(u,r)b_2,\;t^{-s}b_3D(v,s)b_4]$ \\ 
$=\sigma (r,r)\sigma (s,s)[t^{-r}b_1D(u,r)b_2,\;t^{-s}b_3]D(v,s)b_4+\sigma (r,r)\sigma (s,s)t^{-s}b_3[t^{-r}b_1D(u,r)b_2,\;D(v,s)b_4]$\\
$=\sigma (r,r)\sigma (s,s)t^{-r}b_1[D(u,r)b_2,\;t^{-s}b_3]D(v,s)b_4+\sigma (r,r)\sigma (s,s)[t^{-r}b_1,t^{-s}b_3]D(u,r)b_2D(v,s)b_4$\\
$+\sigma (r,r)\sigma (s,s)t^{-s}b_3[t^{-r}b_1,\;D(v,s)b_4]D(u,r)b_2+\sigma (r,r)\sigma (s,s)t^{-s}b_3t^{-r}b_1[D(u,r)b_2\;D(v,s)b_4]$\\
$=-(u,s)\sigma(s,s)t^{-s}b_1b_2b_3D(v,s)b_4+\sigma(r,r)(v,r)t^{-r}b_1b_3b_4D(u,r)b_2$ \\
$+\sigma (r,r)\sigma (s,s)\sigma(s,r)\sigma (r,s)t^{-(r+s)}b_1b_3D(w,r+s)b_2b_4$\\
$=\sigma(r,r)(v,r)t^{-r}b_1D(u,r)b_2b_3b_4-\sigma(s,s)(u,s)t^{-s}b_3D(v,s)b_1b_2b_4$\\
$+\sigma(r+s,r+s)t^{-(r+s)}b_1b_3D(w,r+s)b_2b_4$\\
$=(v,r)T(u,r,b_1,b_2b_3b_4)-(u,s)T(v,s,b_3,b_1b_2b_4)+T(w,r+s,b_1b_3,b_2b_4)$.
\item  
$[D(v,0),t^{-r}b_1D(u,r)b_2]=t^{-r}b_1[D(v,0),D(u,r)b_2]+[D(v,0),t^{-r}b_1]D(u,r)b_2$\\
$=(v,r)t^{-r}b_1D(u,r)b_2-(v,r)t^{-r}b_1D(u,r)b_2=0.$
\item From (2), it follows.
\item Let $U(\tau)=\oplus _{r \in \mathbb{Z}^n}U_r$, where $U_r=\{v\in U(\tau):[D(u,0),v]=(u,r)v,\;\forall u \in \mathbb{C}^n\}$. As $V$ is an irreducible $\tau$-module, using weight argument we get $V_r$ is an irreducible $U_0$-module. Now every element of $U_0$ can be written as a linear combination of the elements $t^{-r_1}b_1D(u_1,r_1)b_2\cdots t^{-r_k}b_{2k-1}D(u_k,r_{k})b_{2k}t^{-s_1}c_1\;ad\;t^{s_1}c_2\cdots t^{-r_k}c_{2k-1}\;ad\;t^{s_k}c_{2k}$, where $r_i\in \textit{rad}\;(f),s_i\notin  \textit{rad}\;(f),b_j,c_j \in B $ with $1\leq i\leq k,\;1\leq j\leq 2k$. So $U_0$ is generated by $T(u,r,b_1,b_2)$ and $t^{-s}b_3\;ad\;t^s b_4$. Hence $V_r$ is an irreducible module for $\mathfrak{g}$. Now $I$ being abelian subalgebra of $\mathfrak{g}$, using Lemma \ref{lemma 3.1}, we can say that $t^{-s}b_1ad\;t^sb_2$ acts as scalar on $V_r$. Therefore $V_r$ will be an irreducible module for $T$.
\end{enumerate}
\end{proof}
Let us define $T^{\prime}(u,r,b_1,b_2)=T(u,r,b_1,b_2)-D(u,0)b_1b_2$ and $T^{\prime}=\textit{span}\{T^{\prime}(u,r,b_1,b_2):r \in \textit{rad}\;f,u \in \mathbb{C}^n,b_1,b_2 \in B\}$. Then we can see that $T^{\prime}$ is a Lie algebra with the Lie brackets \\
$[T^{\prime}(u,r,b_1,b_2),T^{\prime}(v,s,b_3,b_4)]=(v,r)T^{\prime}(u,r,b_1,b_2b_3b_4)-(u,s) T^{\prime}(v,s,b_3,b_1b_2b_4)+ T^{\prime}(w,r+s,b_1b_3,b_2b_4),$ where $w= ((u,s)v-(v,r)u)$. From Proposition(\ref{proposition 3.2})(1) we can see that $[T,T]\subset T^{\prime}$, therefore using Lemma (\ref{lemma 3.1}) and Proposition(\ref{proposition 3.2})(4) one can check that $V_r$ is $T^{\prime}$-irreducible.
\begin{lemma}\label{lemma 3.3}
$V_r\cong V_s$ as $T^{\prime}$-module for all $r,s\in \mathbb{Z}^n$.
\end{lemma}
\begin{proof}
We can see that $[T^{\prime}(u,m,b_1,b_2),t^{(r-s)}]=0$. Now the same map as Proposition (3.4)(5) of \cite{[3]} will give the isomorphism.
\end{proof}
Let $I(u,r,b_1,b_2)=\psi (b_1)\sqrt{\sigma(r,r)}D(u,r)b_2-D(u,0)b_1b_2$ for $u\in \mathbb{C}^n,\; r\in \text{rad} f, \;b_1,b_2\in B$. Consider the space $I=\text{span}\;\{I(u,r,b_1,b_2):u\in \mathbb{C}^n,\; r\in \text{rad} f, \;b_1,b_2\in B\}$. Then $I$ will be a Lie subalgebra of $\tau$ with the Lie brackets
\begin{center}
$[I(u,r,b_1,b_2),I(v,s,b_3,b_4)]=(u,s)I(v,r+s,b_1b_3,b_2b_4)-(v,r)I(u,r+s,b_1b_3,b_2b_4)-(u,s)I(v,s,b_3,b_1b_2b_4)+(v,r)I(u,r,b_1,b_2b_3b_4)$.
\end{center}
Let $\eta\;:I\rightarrow T^{\prime}$ be the map defined as $\eta (I(u,r,b_1,b_2))=T^{\prime}(u,r,b_1,b_2)$.
\begin{lemma}\label{lemma 3.4}
$\eta$ is a Lie algebra isomorphism.
\end{lemma}
Let $W=\textit{Span}\{\sqrt{\sigma(m,m)}t^m.v-v:\; m\in \mathbb{Z}^n,v\in V_0\}=\textit{Span}\{\sqrt{\sigma(m,m)}t^m.v-v:\; m\in \mathbb{Z}^n,v\in V\}$. It can be easily checked that $W$ is a $T^{\prime}$-submodule of $V$. Therefore $V/W$ is a $T^{\prime}$-module.
\begin{lemma}\label{lemma 3.5}
$V/W\cong V_0$ as $T^{\prime}$-module.
\end{lemma}
Now using Lemma \ref{lemma 3.5}, Lemma \ref{lemma 3.4} and Proposition \ref{proposition 3.2}(4) we can say that $V/W$ is irreducible $I$-module.\\
As in \cite{[4]}, consider the space $\widetilde {D}=\textit{span}\{\psi(b_1)D(u,r)b_2-D(u,r)b_1b_2:r\in \textit{rad} \;f,u\in \mathbb{C}^n,b_1,b_2\in B\}$. Then we can see that $\widetilde {D}\subset I$ and $\widetilde{D}=\textit{span}\{D(u,r):u\in \mathbb{C}^n,r\in \textit{rad}\;f\}\otimes M$.
Now as in Lemma 3.3 of \cite{[4]}, we can prove that $D(u,0)b$ acts as scalar (say, $\lambda(u,b)$) on $V_0$ for all $u\in \mathbb{C}^n,b\in B$. We can easily check that $t^k.V_0=V_k,\;\forall k\in \mathbb{Z}^n$. Therefore we have the following for $u\in \mathbb{C}^n,\;k,m\in \mathbb{Z}^n,b\in B,\;v_0\in V_0$.\\
\begin{equation}\label{equation 3.2}
D(u,0)b.t^k.v_0=\{\lambda(u,b)+(u,k)\psi(b)\}t^k.v
\end{equation}
\begin{equation}
t^mb.t^k.v_0=\psi(b)t^m.t^kv_0.
\end{equation}
Now we consider the action of $T^{\prime}(u,r,1,b)$ on $V$ for all $u\in \mathbb{C}^n,\;r(\neq 0)\in \textit{rad}\;(f),b\in B$.
\begin{equation}
T^{\prime}(u,r,1,b).t^k.v_0=\{\sigma(r,r)t^{-r}D(u,r)b-D(u,0)b\}.t^k.v_0
\end{equation}
Now using the claim of Lemma \ref{lemma 3.3} and equation (\ref{equation 3.2}), we will have
\begin{equation}\label{equation 3.5}
D(u,r)b.t^k.v_0=T^{\prime}(u,r,1,b)t^r.t^k.v_0+\{\lambda(u,b)+(u,k)\psi(b)\}t^r.t^k.v_0
\end{equation}
In particular for $b=1$ in equation (\ref{equation 3.5}), we have
\begin{equation}\label{equation 3.6}
D(u,r).t^k.v_0=T^{\prime}(u,r,1,1)t^r.t^k.v_0+(u,k+\alpha)t^r.t^k.v_0
\end{equation}
Now as in \cite{[4]}, for any $r(\neq 0)\in \textit{rad}(f)$, choose $v\in \mathbb{C}^n$ such that $(v,r)\neq 0$ . Now let us consider
\begin{equation}\label{equation 3.7}
[D(v,0)b,D(u,r)]=(v,r)D(u,r)b
\end{equation}
Then using the action of equation \ref{equation 3.2},\ref{equation 3.5} and \ref{equation 3.6} on both sides of \ref{equation 3.7} we have
\begin{equation}\label{equation 3.8}
T^{\prime}(u,r,1,b)=\psi(b)T^{\prime}(u,r,1,1)+\psi(b)(u,\alpha)-\lambda(u,b)
\end{equation}
From (\ref{equation 3.5}), (\ref{equation 3.6}) and (\ref{equation 3.8}) we can see $D(u,r)b=\psi(b)D(u,r)$ on $V$. Infact as in \cite{[4]}, we have $\lambda(u,b)=\psi(b)(u,\alpha)$
\begin{lemma}\label{lemma 3.6}
$\textit{ad}\;t^sb=\psi (b)\textit{ad}\;t^s$ on $V$ for all $s \notin \textit{rad}(f)$.
\end{lemma}
\begin{proof}
By our assumption, the action of  $\mathbb{C}_q^{(1)}\otimes B$ and $\mathbb{C}_q^{(2)}\otimes B$ are associative and anti-associative respectively. Therefore $\textit{ad}\;t^sb=(\textit{ad}\;t^s-t^s)b+t^sb=-(-Ib).(\textit{ad}\;t^s-t^s)+\psi(b)t^s=\psi(b)\textit{ad}\;t^s$
\end{proof}
\begin{theorem}\label{theorem 3.1}
Let $V$ be an irreducible $\tau$-module with finite dimensional weight spaces. Also assume the actions of  $\mathbb{C}_q^{(1)}\otimes B$ and $\mathbb{C}_q^{(2)}\otimes B$ are associative and anti-associative respectively and $t^0$ acts as identity on $V$. Then $V$ is an irreducible module for $\mathfrak{g}$.
\end{theorem}
\begin{proof}
By previous discussion and Lemma \ref{lemma 3.6} we can see it.
\end{proof}
\begin{theorem}\label{theorem 3.2}
Let $V$ be an irreducible $\tau$-module with finite dimensional weight spaces.Also assume the actions of  $\mathbb{C}_q^{(1)}\otimes B$ and $\mathbb{C}_q^{(2)}\otimes B$ are associative and anti-associative respectively and $t^0$ acts as identity on $V$. Then $V\cong F^{\alpha}(V,\psi)$.
\end{theorem}
\begin{proof}
From Theorem \ref{theorem 2.1} and Theorem \ref{theorem 3.1} we have this isomorphism.
\end{proof}


\begin{thebibliography}{99}
\bibitem{[1]} S. Berman, Y. Gao, Y. Kryluk, Quntum tori and structure of elliptic quasi-simple Lie algebras, J. Funct. Anal. 135 (1996) 339-389. 
 \bibitem{[2]} S. E. Rao, K. Zhao, Integrable representations of toroidal Lie algebras co-ordinated by rational quantum tori, J. of Algebra 361 (2012) 225-247.
 \bibitem{[3]} S. E. Rao, P. Batra, S.S. Sharma, The irreducible modules for the derivations of the rational quantum torus, J. of Algebra 410(2014) 333-342.
 \bibitem{[4]} Priyanshu Chakraborty, S. E. Rao, Partial Classification of Irreducible Modules for Loop-Witt Algebras, arXiv:2105.03722
 \bibitem{[5]} S. Eswara Rao, Partial classification of modules for Lie-algebra of diffeomorphisms of d-dimensional torus, Journal of Mathematical Physics 45(8)(2003)
 \bibitem{[6]} James E. Humphreys, Introduction to Lie Algebras and Representation Theory, Grad. Texts in Math., Vol. 9., Springer-Verlag, NewYork, 1978. Second printing, revised.
\end{thebibliography}
\end{document}